\documentclass{amsart}
\usepackage{amsfonts,amssymb,amsmath,amsthm}
\usepackage{url}
\usepackage{enumerate,tikz}

\newtheorem{thrm}{Theorem}[section]
\newtheorem{lem}[thrm]{Lemma}

\newtheorem{cor}[thrm]{Corollary}
\theoremstyle{definition}

\numberwithin{equation}{section}

\newcommand{\R}[1]{\mathbb{R}^{#1}}
\newcommand{\C}[1]{\mathbb{C}^{#1}}
\renewcommand{\S}[1]{\mathbb{S}^{#1}}
\renewcommand{\H}[1]{\mathbb{H}^{#1}}
\newcommand{\im}{\mathrm{im}}
\newcommand{\pd}[2]{\frac{\partial #1}{\partial #2}}
\newcommand{\bd}{\partial}
\newcommand{\ip}[1]{\left\langle #1\right\rangle}
\newcommand{\ud}[2]{\frac{d#1}{d#2}}

\newcommand{\set}[1]{\left\{#1\right\}}
\newcommand{\abs}[1]{\left|#1\right|}
\newcommand{\braces}[1]{\left(#1\right)}
\renewenvironment{cases}{\left\{\begin{aligned}}{\end{aligned}\right.}
\newcommand{\cross}{\times}
\newcommand{\spa}{\hspace{0.86525597943226508722cm}}
\newcommand{\mspa}{\hspace{0.45593812776599623677cm}}

\author{Dylan Cant} \address{
  Mathematics Department\\
  McGill University\\
  Montreal, Quebec} \email{dylan.cant@mcgill.ca} \thanks{This work was
  done with the support of a NSERC student research award.}

\begin{document}

\title[Curvature Flow applied to Isoperimetric Inequality]{A Curvature Flow and Applications to an
  Isoperimetric Inequality}

\begin{abstract}
Long time existence and convergence to a circle is proved for radial graph solutions to a mean curvature type curve flow in warped product surfaces (under a weak assumption on the warp potential of the surface). This curvature flow preserves the area enclosed by the evolving curve, and this fact is used to prove a general isoperimetric inequality applicable to radial graphs in warped product surfaces under weak assumptions on the warp potential.
\end{abstract}
\maketitle

\section{Introduction} \label{sect1}
Extrinsic geometric curvature flows are concerned with evolving closed
hypersurfaces $\Sigma^{n}(t)\subset N^{n+1}$ in the normal direction
with speed equal to some function related to the curvature. A curve
flow is the one-dimensional case $n=1$. A simple example is the curve
shortening flow, which evolves closed curves $\gamma\subset \R{2}$ in
the outward normal direction $\nu$ with speed equal to (minus) the
curvature $\kappa$:
\begin{equation*}
  X_{t}=-\kappa \nu,
\end{equation*}
where $X$ is the position vector of the curve $\gamma$. In
\cite{gage1} (and \cite{gageham}), Gage (and Gage-Hamilton) prove an isoperimetric inequality related to
the curve shortening flow, an idea which was inspiration for the current work
relating isoperimetric inequalities with monotonicity properties of a
curve flow. 

In \cite{MCFSF}, Guan and Li introduce a new type of mean curvature
flow for starshaped hypersurfaces in space forms. This flow is
generalized to starshaped hypersurfaces in warped product spaces in
\cite{MCFWP}. For example, if $\Sigma^{n}$ is a hypersurface in
$\R{n+1}$, then the flow equation is
\begin{equation}\label{eq:1.1}
  \pd{X}{t}=(n-Hu)\nu,
\end{equation}
where $X$ is the position vector for $\Sigma$, $H$ is its mean
curvature, and $u=\ip{X,\nu}$ is the support function of the
hypersurface. For general warped product manifold $N^{n+1}=
(0,R)\cross\S{n}$ with metric $g^{N}:=dr^{2}+\varphi(r)^{2}g^{\S{n}}$,
the \cite{MCFWP} flow equation is
\begin{equation}\label{eq:1intro}
  \pd{X}{t}=(n\varphi'(r)-Hu)\nu, \spa u:=g^{N}(\varphi(r)\bd_{r},\nu).
\end{equation}
It is shown in \cite{MCFSF} that the volume enclosed by the hypersurface is constant along this flow.

When $n\ge 2$, Theorem 1.1 in \cite{MCFSF} guarantees long-time
existence of this flow for smooth initial hypersurfaces, and
exponential convergence to a sphere as $t\to \infty$. Furthermore,
when $n\ge 2$, Theorem 4.2 in \cite{MCFWP} (see also Proposition 3.5 in \cite{MCFSF}) guarantees that, under
the assumption $\varphi'\varphi'-\varphi\varphi''\le 1$, the
surface area of the hypersurface is monotonically decreasing along the
flow; this argument depends on $n\ge 2$ in a crucial way, since it
uses higher elementary symmetric functions of the curvatures:
essentially, the key is the Minkowski identity
\begin{equation}\label{eq:1.3}
  \int_{M}\varphi'H=\frac{2}{n-1}\int_{M}u\sigma_{2}.
\end{equation}
In dimension $n=1$, this inequality is no longer applicable.

The purpose of this current paper is to extend the long-time existence
and monotonicity of surface area results of \cite{MCFSF} to the case
$n=1$, and to prove an isoperimetric inequality for curves in
two-dimensional warped product surfaces with weak assumptions on the
warp potential $\varphi$.

The key results of this paper are:
\begin{thrm}[long-time existence and convergence to a circle]
  Let $N$ be a warped-product space with warp potential $\varphi(r)$
  satisfying $\varphi'\varphi'-\varphi\varphi''\ge 0$. If
  $\gamma_{0}\subset N$ is a smooth hypersurface, then there is a
  unique flow $\gamma(t)$ satisfying \eqref{eq:1intro} and
  $\gamma(0)=\gamma_{0}$. Furthermore $\min r|_{\gamma_{0}}\le
  r|_{\gamma(t)}\le \max r|_{\gamma(0)}$, and $r|_{\gamma(t)}\to
  \text{constant}$ as $t\to \infty$.
\end{thrm}
\begin{thrm}[Isoperimetric inequality]
  If $\gamma_{0}\subset N$ is a piecewise $C^{1}$ Lipshitz radial
  graph, and $\varphi'\varphi'-\varphi\varphi''\in [0,1]$, then
  $L[\gamma_{0}]\ge F(A[\gamma_{0}])$, where equality holds if and
  only if $\gamma_{0}$ is a ``circle,'' where a ``circle'' is either
  \begin{enumerate}[(i)]
  \item a slice $\set{r}\cross \S{1}$.
  \item a translated circle contained in a ``space-form region''
    $[r_{1},r_{2}]\cross \S{1}$ where
    $\varphi'\varphi'-\varphi\varphi''\equiv 1$.
  \end{enumerate}
\end{thrm}
\subsection{Motivating Example}
In this subsection, we show that, in $\R{2}$, the monotonicity of
length (e.g. the ``surface area'' of the hypersurface) of the
one-dimensional curve flow \eqref{eq:1.1} (with $n=1$) is intimately
connected to the classical isoperimetric ratio in $\R{2}$.

Let $\gamma(t)$ be a curve flow which solves \eqref{eq:1.1} with
initial data $\gamma_{0}$. Let $L(t)$, $A(t)$, $ds$, $u$,
$\kappa$ be the length, area (the ``volume''), arc-length one-form,
support function, and curvature of $\gamma(t)$, respectively. As we
will prove in more generality below, we have the following equations
\begin{equation}
  \label{eq:1.4}
  \ud{A(t)}{t}=\int_{\gamma(t)}1-u\kappa\,ds=0\mspa\text{and}\mspa \ud{L(t)}{t}=\int_{\gamma(t)}\kappa-u\kappa^{2}\,ds=2\pi-\int_{\gamma(t)}u\kappa^{2},
\end{equation}
where we have used Gauss-Bonnet in the second equality. 
Define the \emph{isoperimetric difference} $\Lambda:=L^{2}-4\pi
A$. The classical isoperimetric inequality states that $\Lambda\ge 0$,
with equality holding only for circles. We estimate
\begin{equation*}
  L^{2}=\left[\int u^{1/2}u^{1/2}\kappa\,ds\right]^{2}\le\int u\,ds\int
  u\kappa^{2}\le 2A\int u\kappa^{2},
\end{equation*}
where we have used the well-known fact $\int u\,ds=2A$, and the fact
that $\int 1-u\kappa\,ds =0$. From this we can estimate the second
side of \eqref{eq:1.4} to obtain
\begin{equation*}
  2L\ud{L(t)}{t}\le \frac{L}{2A}(4\pi A-L^{2}),
\end{equation*}
which yields 
\begin{equation}\label{eq:1.5}
  \ud{\Lambda}{t}\le -\frac{L}{2A}\Lambda.
\end{equation}
Thus, \emph{if $\Lambda>0$} for our initial curve, then $\Lambda$
decreases as the curve $\gamma(t)$ approaches a circle. By proving
that $\Lambda$ is strictly non-negative in a neighborhood of any
circle, one can use \eqref{eq:1.5} to prove that $\Lambda\not< 0$ for
any initial data $\gamma$. 

Unfortunately, the argument leading to \eqref{eq:1.5} seems to be
special to $\R{2}$, and the author was unable to reproduce it for more
general warped-product spaces. However, a different argument is
possible, and one of the main results in this paper is a general
isoperimetric inequality applicable to a large class of warped product
spaces (subject to a few conditions on the warp potential $\varphi$).
\section{Geometric Preliminaries}
In the remainder of this paper, let $N=(0,R)\cross\S{1}$ be a warped-product surface with metric $g^{N}:=dr^{2}+\varphi(r)^{2}d\theta^{2}$, with $\varphi(r)>0$. Let $\nabla(\cdot,\cdot\cdot)$ be the Levi-Civita connection on $N$.

 Define $\Phi:N\to (0,\infty)$ by 
\begin{equation}
  \label{eq:2.1}
  \Phi(r)=\int_{0}^{r}\varphi(r')\,dr'.
\end{equation}

A closed loop $\gamma\subset N$ is called a $C^{k}$ radial graph if there is a $C^{k}$ mapping $\rho:\S{1}\to (0,R)$ such that $\gamma=\im(\theta\mapsto(\rho(\theta),\theta))$. We denote by $\nu$ and $\bd_{s}$ the outward normal and unit tangent vector of $\gamma$; we assume $\gamma$ has the orientation such that $g(\bd_{r},\bd_{\nu})>0$ and $g(\bd_{\theta},\bd_{s})>0$. The curvature $\kappa$ satisfies $\nabla(\bd_{s},\bd_{\nu})=\kappa\,\bd_{s}$. 
\subsection{Isoperimetric Inequality.}
Let $C_{r}$ denote the circle $\set{r}\cross\S{1}\subset N$, and let $L(r)$ and $A(r)$ denote its length and area, respectively. Then, $A(r)=2\pi \Phi(r)$, and since $\Phi'(r)>0$, we may solve for $r$ as a function of $A$, and consequently, there is some differentiable function $F$ such that
\begin{equation*}
  L(r)^{2}=F(A(r)).
\end{equation*}
The isoperimetric inequality is the statement that, for any curve $\gamma\subset N$,
\begin{equation*}
  L[\gamma]^{2}\ge F(A[\gamma]).
\end{equation*}
Unfortunately, this inequality is not true without some restriction on the warp potential $\varphi$. If $\varphi'\varphi'-\varphi\varphi''>1$, then the inequality will fail. To prove this assertion, we compute the length and area of the radial graph $\gamma(\epsilon)$ with $\rho(\theta)=r_{0}+\epsilon g(\theta)$ to second order in $\epsilon$, and show that
\begin{equation*}
  \begin{aligned} \frac{L[\gamma(\epsilon)]^{2}-F(A[\gamma(\epsilon)])}{\epsilon^{2}}=&4\pi^{2}\left(\left[\frac{1}{2\pi}\int g\,d\theta\right]^{2}+\frac{1}{2\pi}\int (g_{\theta})^{2}-g^{2}\,d\theta\right)\\ &+(\beta-1)\left(\left[\frac{1}{2\pi}\int g\,d\theta\right]^{2}-\frac{1}{2\pi}\int g^{2}\,d\theta\right)
  \end{aligned},
  \end{equation*}
  where $\beta=\varphi'\varphi'-\varphi\varphi''$. If $\beta>1$ and we
  consider $g=\cos\theta$, for example, we obtain $L^{2}-F(A)<0$, to
  lowest order in $\epsilon$. We also note that $\varphi'\varphi'-\varphi\varphi''=1$ for $\R{2},\S{2}$ and $\H{2}$.
\subsection{Geometric properties of Guan-Li mean-curvature flow.}
Following \cite{MCFSF}, we prove the following lemma:
\begin{lem}\label{lem:2.1}
  The vector field $V=\varphi\bd_{r}$ is a conformal Killing vector field with conformal factor $\varphi'$. More precisely, for all vector fields $X,Y$ on $N$, we have
  \begin{equation*}
    \nabla^{2}\Phi(X,Y)=g(\nabla(Y,V),X)=\varphi'g(X,Y).
  \end{equation*}
\end{lem}
\begin{proof}
  Simply compute:
  \begin{equation*} 
    \begin{aligned}
\nabla^{2}\Phi(X,Y)&=Y\ip{d\Phi,X}-\ip{d\Phi,\nabla(Y,X)}=\\ &g(\nabla(Y,V),X)+g(V,\nabla(Y,X))-g(V,\nabla(Y,X)).      
    \end{aligned}
  \end{equation*}
  Then
  \begin{equation*}
    \nabla^{2}\Phi(\bd_{r},\bd_{r})=\varphi'\spa \nabla^{2}\Phi(\bd_{\theta},\bd_{\theta})=\varphi^{2}\Gamma_{\theta r}^{\theta}=\varphi\varphi',
  \end{equation*}
  as desired.
\end{proof}
\begin{cor}
  On a curve $\gamma$, $\Phi_{ss}=\varphi'-u\kappa$. 
\end{cor}
\begin{proof}
  We have $\Phi_{s}=g(V,\bd_{s})$, and taking a second derivative yields
  \begin{equation*}    \Phi_{ss}=g(\nabla(\bd_{s},V),\bd_{s})+g(V,\nabla(\bd_{s},\bd_{s}))=\varphi'-\kappa g(V,\nu)=\varphi'-u\kappa,
  \end{equation*}
  as desired. We used the well-known fact that $\nabla(\bd_{s},\bd_{s})=-\kappa\nu$.
\end{proof}

\begin{thrm}\label{thm:2.2}
Let $\gamma_{0}$, $\gamma(t)$ are closed loops in $N$, and suppose now that $X(\cdot,t):\gamma_{0}\to \gamma(t)$ parametrizes $\gamma(t)$ in terms of $\gamma_{0}$. Suppose that $X$ evolves according to $\bd_{t}X=f\nu$. Then 
\begin{equation*}
  \ud{L}{t}(t)=\int_{\gamma(t)}f\kappa\,ds\mspa\text{and}\mspa \ud{A}{t}(t)=\int_{\gamma(t)}f\,ds,
\end{equation*}
where $L$ and $A$ are the length and area of $\gamma(t)$, respectively.  
\end{thrm}
\begin{proof}
  Consider the normal tube $U_{\epsilon}\simeq \gamma(t_{0})\cross (-\epsilon,\epsilon)$ around $\gamma(t_{0})$ obtained by sending $(p,z)$ to $\mathrm{Exp}_{p}(z\nu(p))$. For $\epsilon$ small enough, this is a diffeomorphism from $\gamma(t_{0})\cross (-\epsilon,\epsilon)$ onto an open set $U_{\epsilon}\subset N$. In the natural coordinates of $U_{\epsilon}$, we can write $g=A(s,z)ds^{2}+dz^{2}$, where $A(s,z)=g(\bd_{s},\bd_{s})$, and so on $\set{z=0}$ we have
  \begin{equation*}
A_{z}(s,0)=g(\nabla(\nu,\bd_{s}),\bd_{s})=g(\nabla(\bd_{s},\nu),\bd_{s})=\kappa.
  \end{equation*}
  Consider $\gamma(t)$ as the graph of $F(\cdot,t):\gamma(t_{0})\to (-\epsilon,\epsilon)$. Clearly $F_{t}(s,t_{0})=f(s)$. Then
  \begin{equation*}
    L(t)=\int_{\gamma(t_{0})}\sqrt{A(s,F(s,t))^{2}+F_{s}^{2}}\,ds,
  \end{equation*}
  and so
  \begin{equation*}
    L_{t}(t_{0})=\int_{\gamma(t_{0})}\kappa(s)f(s)\,ds.
  \end{equation*}
  Similarly,
  \begin{equation*}
    A(t)=\int_{\gamma(t_{0})}\int_{z=0}^{F(s,t)}A(s,z')dz'ds\implies A_{t}(t_{0})=\int_{\gamma(t_{0})}f(s)\,ds.
  \end{equation*}
  This completes the proof.
\end{proof}

\begin{thrm}
  Suppose that $\gamma(t)$ evolves with speed function $f=\Phi_{ss}$. Then
  \begin{equation*}    \ud{A}{t}(t)=0\mspa\text{and}\mspa\ud{L}{t}(t)=\int_{\S{1}}(\varphi'\varphi'-\varphi\varphi'')(r_{s})^{2}-(r_{s\theta})^{2}\,d\theta.
  \end{equation*}
\end{thrm}
\begin{proof}
  The first part of the theorem follows from Theorem \ref{thm:2.2} and fact that $\int\Phi_{ss}\,ds=0$. For the second part of the theorem, we introduce scalar functions $a,b$ defined by
  \begin{equation*}
  \bd_{r}=a\nu+b\bd_{s},\spa\text{(note $r_{s}=b$)}
\end{equation*}
then it is easy to show that
\begin{equation*}
  \bd_{s}=a\frac{\bd_{\theta}}{\varphi}+b\bd_{r}\spa \nu=a\bd_{r}-b\frac{\bd_{\theta}}{\varphi},
\end{equation*}
and using this, we may express the curvature $\kappa$ in terms of $a$ and $b$:
\begin{equation*}
  \kappa=\frac{\varphi'a}{\varphi}-\frac{b_{s}}{a}.
\end{equation*}
It is clear that 
\begin{equation*}
  \Phi_{s}=\varphi b\spa \Phi_{ss}=\varphi' b^{2}+\varphi b_{s},
\end{equation*}
and thus
\begin{equation*}
  \int_{\gamma}\kappa\Phi_{ss}\,ds=\int \frac{a\varphi'}{\varphi}b^{2}-\frac{\varphi}{a}b_{s}^{2}-\frac{\varphi'b^{2}b_{s}}{a}+\varphi'ab_{s}\,ds.
\end{equation*}
Integration by parts on the last term yields
\begin{equation*}
  \int_{\gamma}\kappa\Phi_{ss}\,ds=\int \frac{a\dot{\varphi}^{2}}{\varphi}b^{2}-\frac{\varphi}{a}b_{s}^{2}-\frac{\dot{\varphi}b^{2}b_{s}}{a}-\dot{\varphi}a_{s}b-\frac{a\ddot{\varphi}\varphi}{\varphi}b^{2}\,ds,
\end{equation*}
and thus
\begin{equation*}
  \int_{\gamma}\kappa\Phi_{ss}\,ds=\int \frac{a(\varphi'\varphi'-\varphi\varphi'')}{\varphi}b^{2}-\frac{\varphi}{a}b_{s}^{2}-\frac{\varphi'b^{2}b_{s}}{a}-\varphi'a_{s}b\,ds.
\end{equation*}
now we use the fact that $a^{2}+b^{2}=1$ to deduce that $aa_{s}=-bb_{s}$, which makes the last two terms in the integral cancel, and we are left with
\begin{equation*}
  \int_{\gamma}\kappa\Phi_{ss}\,ds=\int \frac{a(\varphi'\varphi'-\varphi\varphi'')}{\varphi}b^{2}-\frac{\varphi}{a}b_{s}^{2}\,ds
\end{equation*}
Consider the radial graph parametrization of $\gamma$, obtained by
$\theta\mapsto (\theta,\rho(\theta))$. It is straightforward to show that
\begin{equation*}
  \frac{a}{\varphi}ds=d\theta \spa b_{s}=\frac{a}{\varphi}b_{\theta}, 
\end{equation*}
whereby we obtain
\begin{equation}\label{eq:2.2} \ud{L(t)}{t}=\int_{\gamma}\kappa\Phi_{ss}\,ds= \int_{\S{1}}(\varphi'\varphi'-\varphi\varphi'')(r_{s})^{2}(\theta)-(r_{s\theta})^{2}(\theta)\,d\theta,
\end{equation}
where we have replaced $b=r_{s}$. This completes the proof.
\end{proof}
\begin{cor}\label{cor:2.4}
  Suppose that $\varphi'\varphi'-\varphi\varphi''\le 1$. If $\gamma(t)$ satisfies $\int_{\S{1}}r_{s}\,d\theta=0$, then
  $dL(t)/dt\le 0$, with equality if and only if $r_{s}\equiv 0$, or
  $\varphi'\varphi'-\varphi\varphi''\equiv 1$ and $r_{s}=a\cos\theta+b\sin\theta$ ($a,b$ may depend on time).
\end{cor}
\begin{proof} 
  Theorem 2.3 guarantees
  \begin{equation}
    \ud{L(t)}{t}\le \int (r_{s})^{2}-(r_{s\theta})^{2}\,d\theta-2\pi\left[\int r_{s}\,d\theta\right]\le 0,
  \end{equation}  
  where we have used the classical Poincar\'e inequality on the
  circle. The first inequality is equality only when
  $\varphi'\varphi'-\varphi\varphi''\equiv 1$ or $r_{s}\equiv 0$, and
  the second inequality is equality only when $r_{s}=a\cos\theta+b\sin\theta,$ as desired.
\end{proof}
\section{PDE Estimates for Guan-Li curve flow.}
The goal of this section is to prove long-time existence for the curve flow with speed function $\varphi'-u\kappa$ and convergence to circle as $t\to \infty$, assuming smooth radial graph $\gamma_{0}$ as initial data.

Following \cite{MCFSF}, we work in the radial graph parametrization: we look for solutions of the form
\begin{equation}\label{eq:3.1}
  \gamma(t)=\im(\theta\mapsto (\rho(\theta,t),\theta)),
\end{equation}
for some $\rho:\S{1}\cross (0,\infty)\to (0,R)$. Parametrizing the
flow \eqref{eq:3.1} using the radial graph parametrization, we see
that
\begin{equation*}
  \bd_{t}X=\rho_{t}\bd_{r}=\rho_{t}(\frac{u}{\varphi}\nu+r_{s}\bd_{s}),
\end{equation*}
and thus $\gamma(t)$ in \eqref{eq:3.1} evolves with speed function $\rho_{t}u/\varphi$. Thus, if $\rho$ satisfies $\rho_{t}=\varphi f/u$, \eqref{eq:3.1} evolves with speed function $f$. A straightforward computation yields 
\begin{equation*}
  \Phi_{s}=\frac{\varphi \rho_{\theta}}{\sqrt{\varphi^{2}+\rho_{\theta}^{2}}}\spa \Phi_{ss}=\frac{\varphi^{3}\rho_{\theta\theta}+\varphi'\rho_{\theta}^{4}}{(\varphi^{2}+\rho_{\theta}^{2})^{2}},\spa (\text{using }\bd_{\theta}=\sqrt{\varphi^{2}+\rho_{\theta}^{2}}\,\bd_{s}),
\end{equation*}
and $u=\varphi^{2}/\sqrt{\varphi^{2}+\rho_{\theta}^{2}}$.
Therefore, we seek a solution of the following problem
\begin{equation}
  \label{eq:3.2}
  \left\{
      \begin{aligned}
&\rho\in C^{\infty}(\S{1}\cross[0,\infty))\\
&\rho(\theta,0)=\rho_{0}(\theta)\spa (\gamma_{0}=\im(\theta\mapsto(\rho_{0}(\theta),\theta)))\\
&\rho_{t}=\frac{\varphi^{3}\rho_{\theta\theta}+\varphi'\rho_{\theta}^{4}}{\varphi(\varphi^{2}+\rho_{\theta}^{2})^{3/2}}
      \end{aligned}
\right..
\end{equation}
\subsection{\textit{C}$^{\mathbf{0}}$ and gradient estimate for solutions of \eqref{eq:3.2}}
Following \cite{MCFSF}, we first prove solutions of \eqref{eq:3.2} satisfy a $C^{0}$ a priori estimate. At critical points of $\rho$ we have $\rho_{\theta}=0$ and thus
\begin{equation*}
  \rho_{t}=\frac{1}{\varphi}\rho_{\theta\theta},
\end{equation*}
and, by the standard maximum principle, this implies that
\begin{equation}\label{eq:3.3}
  \min\rho_{0}\le\rho(\theta,t)\le \max \rho_{0}.
\end{equation}
The gradient estimate requires more work. We begin by showing that $\omega:=\rho_{\theta}^{2}$ satisfies a parabolic evolution equation. We compute
\begin{equation}
  \begin{aligned} \omega_{t}=2\rho_{\theta}\rho_{\theta,t}&=\frac{\varphi^{3}}{\varphi(\varphi^{2}+\rho_{\theta}^{2})^{3/2}}\omega_{\theta\theta}-2 \frac{\varphi^{3}}{\varphi(\varphi^{2}+\rho_{\theta}^{2})^{3/2}}(\rho_{\theta\theta})^{2}\\+&\left[\pd{}{\theta}\braces{\frac{\varphi^{2}}{(\varphi^{2}+\rho_{\theta}^{2})^{3/2}}}+\frac{2\varphi'\omega}{\varphi(\varphi^{2}+\rho_{\theta}^{2})^{3/2}}-\frac{3\varphi\rho_{\theta}^{5}}{\varphi(\varphi^{2}+\rho_{\theta}^{2})^{5/2}}\right]\omega_{\theta}\\ &\hphantom{blah blah blah blah blah blah blah}-\frac{2(4\varphi'\varphi'- \varphi''\varphi)\rho_{\theta}^{6}}{(\varphi^{2}+\rho_{\theta}^{2})^{5/2}}-\frac{2(\varphi'\varphi'-\varphi''\varphi)\rho_{\theta}^{8}}{(\varphi^{2}+\rho_{\theta}^{2})^{5/2}},
  \end{aligned}  
  \end{equation}
  and abbreviating yields
  \begin{equation}    \omega_{t}=A(\rho,\omega)\omega_{\theta\theta}-2A(\rho,\omega)(\rho_{\theta\theta})^{2}+B(\rho,\omega,\omega_{\theta})\omega_{\theta}-C_{1}(\rho,\omega)\omega^{3}-C_{2}(\rho,\omega)\omega^{4}
  \end{equation}
  We now make the assumption that
  $\varphi'\varphi'-\varphi''\varphi> 0$, which enables us to
  conclude that $C_{1}> 0$ and $C_{2}> 0$. We remark that this
  assumption also plays a key role in proving the convergence to a circle. Supposing that $\omega$ achieves a
  positive maximum at $(\theta,t)$ for a positive time $t>0$, we have
  \begin{equation*}
    \omega_{t}(\theta,t)=A(\rho,\omega)\omega_{\theta\theta}(\theta,t)-C_{1}(\rho,\omega)\omega^{3}-C_{2}(\rho,\omega)\omega^{4},
1  \end{equation*}
  which implies that $\omega_{t}<0$, a contradiction. Thus $\omega$
  can only attain its maximum on the initial data, so
  \begin{equation}
    0\le \omega(t>0)<\omega(t=0).
  \end{equation}
  \subsection{Long-time existence of solutions of equation \eqref{eq:3.2}.}
  Appealing to the classical parabolic theory for quasi-linear
  parabolic equations (see, for instance, \cite{pdetext}), the higher
  regularity a priori estimates for $\rho$ follow from the uniform
  $C^{0}$ and gradient estimates. To be precise, we have:
  \begin{thrm}
    Let $\rho$ solve \eqref{eq:3.2}, and suppose that $\varphi'\varphi'-\varphi\varphi''>0$. Then, for any integer $k\ge 0$, and any $t_{0}>0$, there exists a
    constant $C(t_{0},\rho_{0},k)$ such that
    \begin{equation}
      \label{eq:3.7}
      \abs{\rho(\cdot,t)}_{C^{k}}<C(t_{0},\rho_{0},k)\spa t>t_{0}.
    \end{equation}
  \end{thrm}
  These estimates guarantees
  long-time existence and uniqueness and thus the first part of
  Theorem 1.1 (see, for instance, the proof of Theorem 8.3 in
  \cite{Lieberman}).
 \subsection{Convergence to a circle as $t\to\infty$.}
 The goal of this subsection is to prove the second part of Theorem
 1.1, that $r|_{\gamma(t)}\to \text{constant}.$
 
 Suppose that $\rho$ is a solution to \eqref{eq:3.2}. As shown in the
 previous section, $\omega=\rho_{\theta}^{2}$ satisfies the following
 \textsc{pde}
 \begin{equation*}
   \omega_{t}=A(\rho,\omega)\omega_{\theta\theta}-2A(\rho,\omega)(\rho_{\theta\theta})^{2}+B(\rho,\omega,\omega_{\theta})\omega_{\theta}-C_{1}(\rho,\omega)\omega^{3}-C_{2}(\rho,\omega)\omega^{4},
 \end{equation*}
 with constants $a_{i},b_{1},c_{i,j}$, such that
 $0\le a_{1}\le A\le a_{2}<\infty$,
 $0< c_{i,1}\le C_{i}\le c_{i,2}<\infty$, $i=1,2$. 

Let $v:[0,\infty)\to [0,\infty)$ be the unique solution to
\begin{equation*}
  v_{t}=-c_{1,1}v^{3}-c_{2,1}v^{4}\spa v(0)=\max \omega+\epsilon\spa (\epsilon>0)
\end{equation*}
It is clear that $v$ is always positive, and that $v=O(t^{-1/2})$. 

Letting $u=\omega-v$, which satisfies
\begin{equation}\label{eq:1231}
  \begin{aligned}
  u_{t}\le A(\rho,\omega)&u_{\theta\theta}-2A(\rho,\omega)(u_{\theta\theta})^{2}+B(\rho,\omega,\omega_{\theta})u_{\theta}\\&-c_{1,1}u(\omega^{2}+v\omega+v^{2})-c_{2,1}u(\omega^{3}+v\omega^{2}+v^{2}\omega+v^{3}).    
  \end{aligned}
\end{equation}
Suppose now that $u$ achieves a positive maximum at $(\theta,t)$. Then
evaluating the right hand side of \eqref{eq:1231} yields
$u_{t}(\theta,t)<0$, which is a contradiction. Thus $u$ can not
achieve a positive maximum, and so $\omega\le v$, and since $v\to 0$
uniformly as $t\to \infty$, we deduce that $\omega\to 0$ uniformly,
and thus $\rho(\cdot,t)$ uniformly approaches a constant function as
$t\to \infty$. This proves that the curve flow converges to a circle.

\section{Monotonicity of length and the Isoperimetric inequality}
Let $\gamma(t)$ be a solution to the flow with speed function
$\varphi'-u\kappa$. Unfortunately, it seems difficult to directly
prove that $L[\gamma(t)]$ is monotonically nonincreasing - however, we
will prove that $L[\gamma(t)]$ is monotonically nonincreasing for
\textbf{bilaterally symmetric} curves:
\subsection{Bilaterally symmetric radial graphs}
  For each $\alpha\in \S{1}$, consider the isometry $\mathfrak{R}_{\alpha}:N\to N$ defined by
\begin{equation*}
  \mathfrak{R}_{\alpha}:(r,\theta)\mapsto (r,\alpha^{2}\theta^{-1}),
\end{equation*}
where multiplication happens in $\S{1}\subset \C{}$. Note that the antipodal points $\pm\alpha$ are fixed, and so we can consider this as a reflection through a line:
\begin{equation*}
  \begin{tikzpicture}[rotate=30]
    \draw[->] (1.5,0) arc (0:-330:0.1 and 0.3);
    \draw[white,thick] (1.5,0) arc (0:-30:0.1 and 0.3);

    \draw (-2,0) -- (2,0);
    \draw[variable=\t,domain=180:360,samples=70] plot ({cos(\t)},{sin(\t)+0.3*sin(2*\t)*sin(2*\t)-0.3*sin(2*\t)*sin(4*\t)});
    \draw[variable=\t,domain=180:360,samples=70] plot ({cos(\t)},{-(sin(\t)+0.3*sin(2*\t)*sin(2*\t)-0.3*sin(2*\t)*sin(4*\t))});
    \node at (2,0) [right] {$\alpha$};
  \end{tikzpicture}
\end{equation*}
It is clear that $\mathfrak{R}_{\alpha}:N\to N$ is an isometry. If
$\gamma\subset N$ is a curve such that there is some $\alpha\in \S{1}$
such that $\gamma$ is fixed under $\mathfrak{R}_{\alpha}$, then we say
$\gamma$ has a \textbf{bilateral symmetry} with axis $\alpha$.

\begin{thrm} Let $0\le \varphi'\varphi'-\varphi\varphi''\le 1$.
  Suppose $\gamma_{0}$ is a smooth, bilaterally symmetric with axis
  $\alpha$, radial graph, and suppose $\gamma(t)$ is a solution to the
  flow with initial data $\gamma(0)=\gamma_{0}$. Then $\gamma(t)$ is
  also bilaterally symmetric with axis $\alpha$ and furthermore
  $L[\gamma(t)]$ is nonincreasing.
\end{thrm}
\begin{proof}
  The fact that $\gamma(t)$ is also bilaterally symmetric follows from
  uniqueness of solutions and the fact that
  $\mathfrak{R}_{\alpha}(\gamma(t))$ also is a solution to the
  flow. 

  If $\gamma$ is bilaterally symmetric radial graph, and
  $p\in \gamma$, then it is an easy computation to show that
  $r_{s}(p)=-r_{s}(\mathfrak{R}_{\alpha}(p))$, and so
  $\int_{\S{1}}r_{s}\,d\theta=0$. Then Corollary 2.4 implies that
  $\ud{}{t}L[\gamma(t)]\le 0$, as desired. 
\end{proof}
\begin{cor}
   Let $0\le \varphi'\varphi'-\varphi\varphi''\le 1$. If $\gamma_{0}$ is a smooth bilaterally symmetric radial graph, then
  $L[\gamma_{0}]\ge F(A[\gamma_{0}])$ (cf. subsection 2.1 for
  definition of $F$).
\end{cor}
\begin{proof}
  Using monotonicity of length and the constancy of area, we deduce
  $L[\gamma_{0}]\ge
  L[\gamma(\infty)]=F(A[\gamma(\infty)])=F[A(\gamma_{0})]$, as
  desired. 
\end{proof}
\begin{thrm}[Isoperimetric inequality, without equality case]\label{thrm:4.3}
 Let $0\le \varphi'\varphi'-\varphi\varphi''\le 1$. If $\gamma_{0}$ is any piecewise $C^{1}$ and Lipshitz radial graph, then
  $L[\gamma_{0}]\ge F[A(\gamma_{0})]$.
\end{thrm}
\begin{proof}
  Let $\rho_{0}$ be the radial length function of $\gamma_{0}$, and, for each $\alpha$, define
  \begin{equation*}
    \rho_{\alpha,1}(\theta)=
    \begin{cases}
      \rho(\theta)\spa \text{imag}(\alpha\theta^{-1})\ge 0\\
      \rho(\alpha^{2}\theta^{-1})\spa
      \text{imag}(\alpha\theta^{-1})\le 0.
    \end{cases}\spa \text{considering $\alpha,\theta\in \S{1}\subset\C{}$}
  \end{equation*}
  \begin{equation*}
    \rho_{\alpha,2}(\theta)=
    \begin{cases}
      \rho(\theta)\spa \text{imag}(\alpha^{-1}\theta)\le 0\\
      \rho(\alpha^{2}\theta^{-1})\spa
      \text{imag}(\alpha^{-1}\theta)\ge 0.
    \end{cases}
  \end{equation*}
Then note that $\text{imag}(\alpha^{-1}\theta)\ge 0\ (\le 0)$ implies that
$$\text{imag}(\alpha^{-1}\alpha^2\theta^{-1})=\text{imag}(\alpha\theta^{-1})\le
0 \ (\ge 0),$$
and so 
$\rho_{\alpha,1}(\alpha^{2}\theta^{-1})=\rho_{\alpha,1}(\theta)$ and
$\rho_{\alpha,2}(\alpha^{2}\theta^{-1})=\rho_{\alpha,2}(\theta)$.
\begin{figure}[h]
  \centering
  \includegraphics{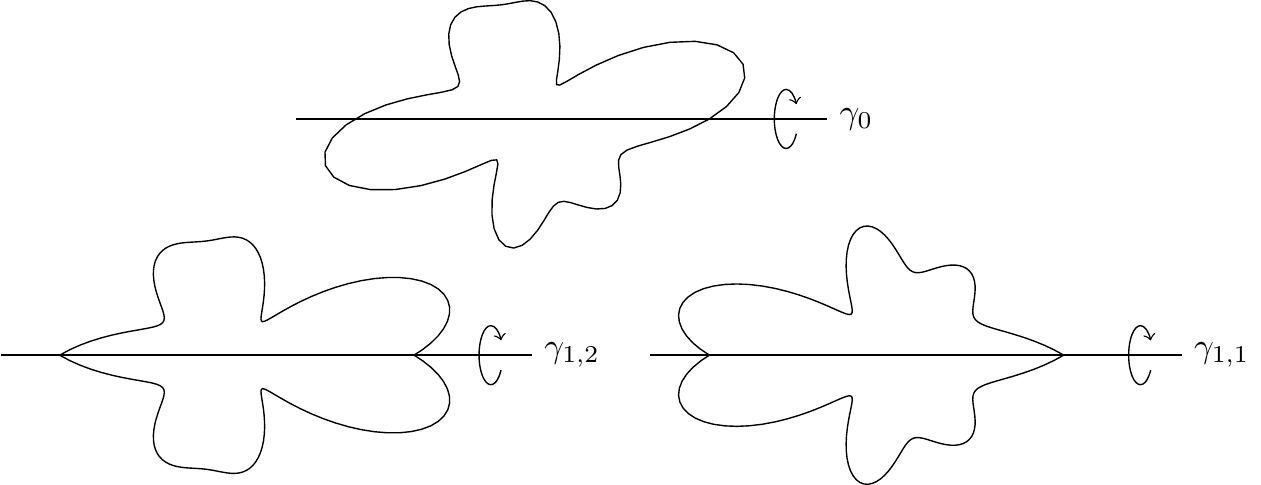}
\end{figure}
Let $\gamma_{0},\gamma_{\alpha,1}$ and $\gamma_{\alpha,2}$ be the curves defined by
$\rho_{0}$, $\rho_{\alpha,1}$, $\rho_{\alpha,2}$. By continuity there is some $\alpha\in \S{1}$ such that
$\gamma_{\alpha,1}$ and $\gamma_{\alpha,2}$ have the same
area. Suppose that $L[\gamma_{\alpha,1}]\le L[\gamma_{0}]\le
L[\gamma_{\alpha,2}]$. It therefore suffices proving that
$L[\gamma_{\alpha,1}]\ge F[A(\gamma_{\alpha,1})]=F[A(\gamma_{0})]$,
and so we may assume from the outset that $\rho_{0}$ is piecewise
$C^{1}$, Lipshitz, and bilaterally symmetric.

In this case, by convolution with mollifiers $f_{n}:\S{1}\to \R{}$
symmetric about $1\in \S{1}$, let $\rho_{n}\to \rho_{0}$ be a sequence
of smooth, symmetric, functions converging to $\rho_{0}$. Since
$\rho_{0}$ is Lipshitz, its derivative is bounded. Since
$\rho_{n,\theta}\to \rho_{0,\theta}$ uniformly on any 
compact set $K\subset \S{1}$ such that $\rho_{0}\in C^{1}(K)$, we
deduce that $L[\gamma_{n}]\to L[\gamma_{0}]<\infty$ and $A[\gamma_{n}]\to
A[\gamma_{0}]<\infty$. Then $L[\gamma_{0}]\ge F(A[\gamma_{0}])$
follows, as desired.
\end{proof}
\section{Equality case for isoperimetric inequality.}
The rest of this paper is dedicated to discussing when equality holds
in Theorem \ref{thrm:4.3}. Ideally, $L[\gamma_{0}]=F(A[\gamma_{0}])$ would imply
that $\gamma_{0}$ is a circle $\set{r}\cross\S{1}$, but, when $N$
satisfies $\varphi'\varphi'-\varphi\varphi''=1,$ there are
complications arising from translation isometries of $N$
(e.g. translated circles in the space forms $\R{2}$, $\S{2}$ and
$\H{2}$ are equality cases).

A key tool in our argument is the following existence result:
\begin{thrm}\label{thrm:5.1}
  Suppose $\varphi'\varphi'-\varphi\varphi''\ge 0$. Let $\gamma_{0}$
  be a piecewise $C^{1}$ Lipshitz radial graph. Then there is a smooth
  solution $\gamma(t)$ to the flow $\varphi'-u\kappa$ such that
  $\rho_{\gamma(t)}(\cdot,t)\to \rho_{\gamma(0)}(\cdot)$ uniformly,
  $L[\gamma(t)]\to L[\gamma_{0}]$, and $A[\gamma(t)]\to A[\gamma_{0}]$
  as $t\searrow 0$. If $\gamma_{0}$ is bilaterally symmetric, then we
  may take $\gamma(t)$ to be bilaterally symmetric as well.
\end{thrm}
\begin{proof}
  As in the proof of Theorem \ref{thrm:4.3}, convolve with smooth mollifiers to
  obtain $\rho_{n}\to \rho_{0}$ a sequence of smooth radial length
  functions converging to $\rho_{0}$. If $\gamma_{0}$ (hence
  $\rho_{0}$) has a bilateral symmetry, we may take $\rho_{n}$ to have
  the same bilateral symmetry. Using the existence result Theorem 1.1 for smooth initial data, let $\rho_{n}(\cdot,t)$ be solutions to \eqref{eq:3.2} with initial data $\rho_{n}$. For each $t_{0}>0$, Theorem 3.1 guarantees that there is a constant $C(t_{0},\rho_{0},k)$ such that 
  \begin{equation}\label{eq:5.1}
    \abs{\rho_{n}(\cdot,t)}_{C^{k}}\le C(t_{0},\rho_{0},k)\spa t\ge t_{0}.
  \end{equation}
  We can take $C$ to be independent of $n$, since it only depends on the $C_{0}$ and $C_{1}$ estimates of the initial data $\rho_{n}$, and for $n$ sufficiently large, the $C_{0}$ and $C_{1}$ estimates of $\rho_{n}$ can be estimated from the $C_{0}$ and Lipshitz estimates of $\rho_{0}$.

  Now recursively define subsequences of $(\rho_{n}:n\in \mathbb{N})$ by the two properties
  \begin{enumerate}[(i)]
  \item $(\rho_{mn}:n\in \mathbb{N})$ converges to a smooth limit
    function on $[1/m,m]\cross \S{1}$ satisfying \eqref{eq:3.2} (convergence in $C^{k}$ for every $k$).
  \item $(\rho_{mn}:n\in \mathbb{N})$ is a subsequence of $(\rho_{m'n}:n\in \mathbb{N})$ if $m'>m$.
  \end{enumerate}
  It is always possible to do this because $\set{\rho_{n}:n\in
    \mathbb{N}}$ satisfies the (uniform) estimates \eqref{eq:5.1}
  (here we are using Arz\'ela-Ascoli, invoking the fact that
  $\rho_{n}:n\in \mathbb{N}$ is bounded in $C^{k+1}(\S{1}\cross[1/m,m])$).

Then, using the standard diagonal trick, $\rho_{mm}$ converges to a
smooth function $\rho$ on $C^{k}(\S{1}\cross(0,\infty))$ which
satisfies \eqref{eq:3.2}. By taking more subsequences, we may upgrade
this to $C^{\infty}(\S{1}\cross (0,\infty))$ convergence. It is clear that $\rho(\cdot,t)\to \rho_{0}(\cdot)$ uniformly and such that the length and area are continuous as $t\searrow 0$, as desired.
\end{proof}

\begin{thrm}\label{thrm:5.2}
  Let $\gamma_{0}\subset N$ be a piecewise $C^{1}$ Lipshitz radial
  graph. Suppose that $\varphi'\varphi'-\varphi\varphi''\in [0,1]$ and
  $\varphi'\varphi'-\varphi''\varphi\not\equiv 1$ on $\gamma_{0}$. If
  $L[\gamma_{0}]=F(A[\gamma_{0}])$, then $\gamma_{0}$ is a circle
  $\set{r}\cross \S{1}$.
\end{thrm}
\begin{proof}
  First, using the ``cut and reflect'' technique in Theorem 4.3, we
  may cut $\gamma_{0}$ into two bilaterally symmetric halves
  $\gamma_{1}$ and $\gamma_{2}$ satisfying
  $A[\gamma_{1}]=A[\gamma_{2}]=A[\gamma_{0}]$. Clearly we have
  $L[\gamma_{1}]=F(A[\gamma_{1}])$. Now apply Theorem 5.1 to deduce a
  solution to the flow $\gamma_{1}(t)$. It is clear that
  $L[\gamma_{1}(t)]$ is a constant. By Corollary 2.4, we deduce that
  $r_{s}(t)\equiv 0$, and so the $\gamma_{1}(t)=\set{r}\cross \S{1}$
  (for all $t>0$). This is obviously stable as $t\to 0$ by uniform
  convergence, so $\gamma_{1}=\set{r}\cross\S{1}$. Clearly $\gamma_{2}=\set{r}\cross \S{1}$ for the \emph{same} $r$ (by continuity), so $\gamma_{0}=\set{r}\cross \S{1}$, as desired.
\end{proof}
Now we must consider the case when $\varphi'\varphi'-\varphi\varphi''=1$. First, we note the following result:
\begin{lem}\label{lem:5.3}
  If $N=(r_{1},r_{2})\cross \S{1}$ is a warped-product space satisfying $\varphi'\varphi'-\varphi\varphi''\equiv 1$, then there is $k>0$ and $r_{0}$ such that $\varphi(r)=\sinh(k(r-r_{0})),\sin(k(r-r_{0}))$ or $r-r_{0}$, and $N$ can be isometrically embedded into $k^{-1}\H{2}$, $k^{-1}\S{2}$ or $\R{2}$. (The gauss curvature is $-k^{2},k^{2},0$). 
\end{lem}
\begin{proof}
We use the well known fact that $k^{-1}\H{2}-\set{\text{origin}}$,
$k^{-1}\S{2}-\set{\pm (0,0,1)}$ and
$\R{2}-\set{(0,0)}$ are warped product surfaces with potentials $\sinh
kr$, $\sin kr$ and $r$, respectively. 

Note that $\varphi'\varphi'-\varphi''\varphi\equiv 1$
  implies that the gauss curvature $K=-\varphi''/\varphi$ is constant,
  and so 
  \begin{equation*}
    \varphi(r)=
    \begin{cases}
      A\sinh(k(r-r_{0})) \mspa K=-k^{2}\\
      A\sin(k(r-r_{0}))\mspa K=k^{2}\\
      Ar-r_{0}\mspa K=0
    \end{cases}.
  \end{equation*}
  Then invoking $\varphi'\varphi'-\varphi\varphi''=1$, we deduce that
  $A=1$ in all cases. Then the map
  $(r_{1},r_{2})\cross \S{1}\to (0,\infty)\cross \S{1}$ defined by
  $(r,\theta)\mapsto (r-r_{0},\theta)$ is an isometric embedding of
  the tube into $k^{-1}\mathbb{H}^{2}$, $k^{-1}\mathbb{S}^{2}$,
  or $\R{2}$, depending on the sign of $K$.
\end{proof}
Thanks to this result, if we are given $\gamma_{0}\subset N$ satisfying $\varphi'\varphi'-\varphi\varphi''\equiv 1$ on $\gamma_{0}$, then we may assume that $\gamma_{0}\subset k^{-1}\H{2},k^{-1}\S{2}$ or $\R{2}$, and then, up to rescaling, $\gamma_{0}\subset \H{2},\S{2}$ or $\R{2}$. If $\gamma_{0}$ is a radial graph in $N$, then we may assume it is a radial graph in $\H{2}$, $\S{2}$ or $\R{2}$. 

Referring to Corollary 2.4, if $\gamma(t)$ is a bilaterally symmetric solution to the curve flow (in one of these three space forms) satisfying $L[\gamma(t)]=\text{constant}$, then $r_{s}=a\cos\theta+b\sin\theta$ on $\gamma(t)$. We will show that this condition on the radial length function $r$ implies that $\gamma(t)$ is a circle (which may be translated). 

We first show that curves satisfying $r_{s}=a\cos\theta+b\sin\theta$ are unique up to initial point.
\begin{thrm}\label{thrm:5.4}
  Suppose that $\gamma_{1},\gamma_{2}\subset N$ are $C^{1}$ radial graphs in a
  warped-product space $N$ and $\gamma_{1}\cap \gamma_{2}\ne \emptyset$. Suppose that the unit tangent fields $\bd_{1,s}$ and $\bd_{2,s}$ satisfy $g^{N}(\bd_{\theta},\bd_{j,s})>0$, $j=1,2$.
  If $r_{s}=a\cos(\theta+\alpha)$ is true on both
  $\gamma_{1}$ and $\gamma_{2}$ for the same constants $a\in(-1,1)$ and $\alpha\in \R{}/2\pi\mathbb{Z}$, then $\gamma_{1}=\gamma_{2}$.  
\end{thrm} 
\begin{proof}
  Let $p\in \gamma_{1}\cap\gamma_{2}$ and let $F_{j}:\R{}\to \gamma_{j}$ be the unique (arc-length)
  parametrization satisfying
  \begin{equation*}
    F_{j}(0)=p\spa F_{j}'(s)=\bd_{j,s},
  \end{equation*}
  where $\bd_{j,s}$ is the unit tangent vector to $\gamma_{j}$. Let
  $F_{j}=(r_{j},\theta_{j})$, where $\theta_{j}:\R{}\to
  \R{}/2\pi\mathbb{Z}$. Using the orthogonal decomposition
  \begin{equation*}
    \bd_{j,s}=g^{N}(\bd_{r},\bd_{j,s})\bd_{r}+g^{N}(\frac{1}{\varphi}\bd_{\theta},\bd_{j,s})\frac{\bd_{\theta}}{\varphi}=r_{j}'(s)\bd_{r}+\varphi\theta_{j}'(s)
    \frac{\bd_{\theta}}{\varphi},
  \end{equation*}
  we deduce that
  \begin{equation*}
    \theta_{j}'(s)^{2}=\frac{1-r_{j}'(s)^{2}}{\varphi^{2}(r_{j}(s))},
  \end{equation*}
  and thus
  \begin{equation}\label{eq:5.2}
    r_{j}(0)=r(p)\mspa\theta_{j}(0)=\theta(p)\mspa
    r_{j}'=a\cos(\theta_{j}+\alpha)\mspa
    \theta_{j}'=\frac{\sqrt{1-a^{2}\cos^{2}(\theta_{j}+\alpha)}}{\varphi(r_{j})}.
  \end{equation}
  We are allowed to choose the positive square root since $g^{N}(\bd_{\theta},\bd_{j,s})>0$ implies that $\theta_{j}'>0$.

  Since solutions to \eqref{eq:5.2} are unique, we conclude that
  $F_{1}=F_{2}$, and thus $\gamma_{1}=\gamma_{2}$, as desired.
\end{proof}
To prove that the only graphs satisfying $r_{s}=a\cos\theta$ are
circles, we find it convenient to split the argument into three
sections depending on the ambient space $\R{2}$, $\S{2}$ or
$\H{2}$. 
\begin{lem}
  Given any $a\in (-1,1)$ and $\alpha\in \R{}/2\pi\mathbb{Z}$, the circle
  \begin{equation}\label{eq:3}
    C(a,\alpha,R):=\im(\beta\mapsto(aR\sin \alpha+R\cos\beta,aR\cos\alpha+R\sin\beta))
  \end{equation}
  satisfies $r_{s}=a\cos(\theta+\alpha)$.
\end{lem}
\begin{proof}
  By rotational symmetry of $\R{2}$ it suffices to prove the case
  $\alpha=0$. Then we note that
  \begin{equation*}
    r_{s}=a\cos(\theta)\iff rr_{s}=ax\iff xx_{s}+yy_{s}=ax,
  \end{equation*}
  plugging in $x(\beta),y(\beta)$, we obtain
  \begin{equation*}    xx_{s}+yy_{s}=-R\cos\beta\sin\beta+aR\cos\beta+R\cos\beta\sin\beta=aR\cos\beta=ax,
  \end{equation*}
  and so indeed $r_{s}=a\cos(\theta)$ on the circle
  $\im(\beta\mapsto(R\cos\beta,aR+R\sin\beta)),$
  which completes the proof.
\end{proof}
\begin{thrm}\label{thrm:5.6}
    If $\gamma\subset \R{2}$ is a $C^{1}$ radial graph which satisfies $r_{s}=a\cos(\theta+\alpha),$ then $\gamma$ is a circle.
  \end{thrm}
  \begin{proof}
    First note that if $\abs{a}\ge 1$, then $\abs{r_{s}(\theta=-\alpha)}\ge 1$ which contradicts the fact that $r_{\nu}>0$, since $r_{s}^{2}+r_{\nu}^{2}=1$. Thus $a\in (-1,1)$. This argument uses the fact that $\gamma$ is a radial graph to deduce that $r_{\nu}>0$ and that there exists some point on $\gamma$ with $\theta=-\alpha$.

    Consider the circles $C(a,\alpha,R)$ defined in the previous theorem. It is straightforward to see that the two points
    \begin{equation*}
      \begin{aligned}
      c_{+}(a,\alpha,R):=((aR+R)\sin \alpha,(aR+R)\cos\alpha)\\ c_{-}(a,\alpha,R):=((aR-R)\sin \alpha,(aR-R)\cos\alpha)        
      \end{aligned}
    \end{equation*}
    both lie on $C(a,\alpha,R)$. Since $\abs{a}<1$, $R\mapsto c_{+}(R)$ is surjective onto the ray $\R{}_{+}(\sin\alpha,\cos\alpha)$ and $R\mapsto c_{-}(R)$ is surjective onto the ray $\R{}_{-}(\sin\alpha,\cos\alpha)$.

    Let $p$ be the point on $\gamma$ where the radius function
  $r$ is maximized. Then $r_{s}(p)=0$, so
  $\theta(p)+\alpha=\pi/2+\pi\mathbb{Z}$. It follows that $\gamma$ intersects
  the line $\R{}(\sin(\alpha),\cos(\alpha))$, and so there is some $R$ such that either $p=c_{+}(a,\alpha,R)$ or $p=c_{-}(a,\alpha,R)$. Theorem \ref{thrm:5.4} implies that $\gamma=C(a,\alpha,R)$.
  \end{proof}

Turning now to the $\S{2}$ case, we consider $\S{2}\subset \R{3}$ with
$x=\sin r \cos\theta$, $y=\sin r\sin\theta$ and $z=\cos r$. 
\begin{thrm}\label{thrm:5.7}
  The circle $C_{R}(p)\subset \S{2}-\set{(0,0,\pm 1)}$, $R\ne 0$, satisfies 
$$r_{s}=\frac{y(p)}{\sin R}\cos\theta-\frac{x(p)}{\sin R}\sin \theta$$
\end{thrm}
\begin{proof}
  It suffices to prove that
  \begin{equation}\label{eq:5.4}
    \sin r\, r_{s}=\frac{y(p)}{\sin R}\sin r\cos\theta-\frac{x(p)}{\sin R}\sin r\sin\theta \iff z_{s}=\frac{x(p)y-y(p)x}{\sin R},
  \end{equation}
  where we have used $x=\sin r \cos\theta$, $y=\sin r\sin\theta,$ and $z=\cos r$.

  To prove this, we note that, at a point $q$ on the circle $C_{R}(p)$, the unit tangent vector $\bd_{s}(q)$ satisfies 
  \begin{equation*}
    \bd_{s}(q) \sin R= p\cross q \mspa\text{(vector cross product),}
  \end{equation*}
  which implies that
  \begin{equation*}
    z_{s}\sin R=\det
    \begin{pmatrix}
      x(p)&y(p)\\
      x(q)&y(q)
    \end{pmatrix},
  \end{equation*}
  which is equivalent to \eqref{eq:5.4}.
\end{proof}
\begin{thrm}\label{thrm:5.8}
  If $\gamma\subset \S{2}-\set{(0,0,\pm 1)}$ is a radial graph which satisfies $r_{s}=a\cos\theta+b\sin\theta$, then $\gamma$ is a circle.
\end{thrm}
\begin{proof}
  For simplicity, rotate $\gamma$ so that $r_{s}=b\sin\theta$, with
  $b>0$. As in the proof of Theorem \ref{thrm:5.6} We may assume that $b<1$. There is unique smooth $f:(0,\pi)\to(0,\pi)$ such that $b\sin R=\sin f(R)$; clearly $f(R)<R$. Now consider the circle $C_{R}(p)$ where $p=(\sin f(R),0,\cos f(R))$. 
  \begin{figure}[h]
    \centering
\includegraphics[scale=0.3]{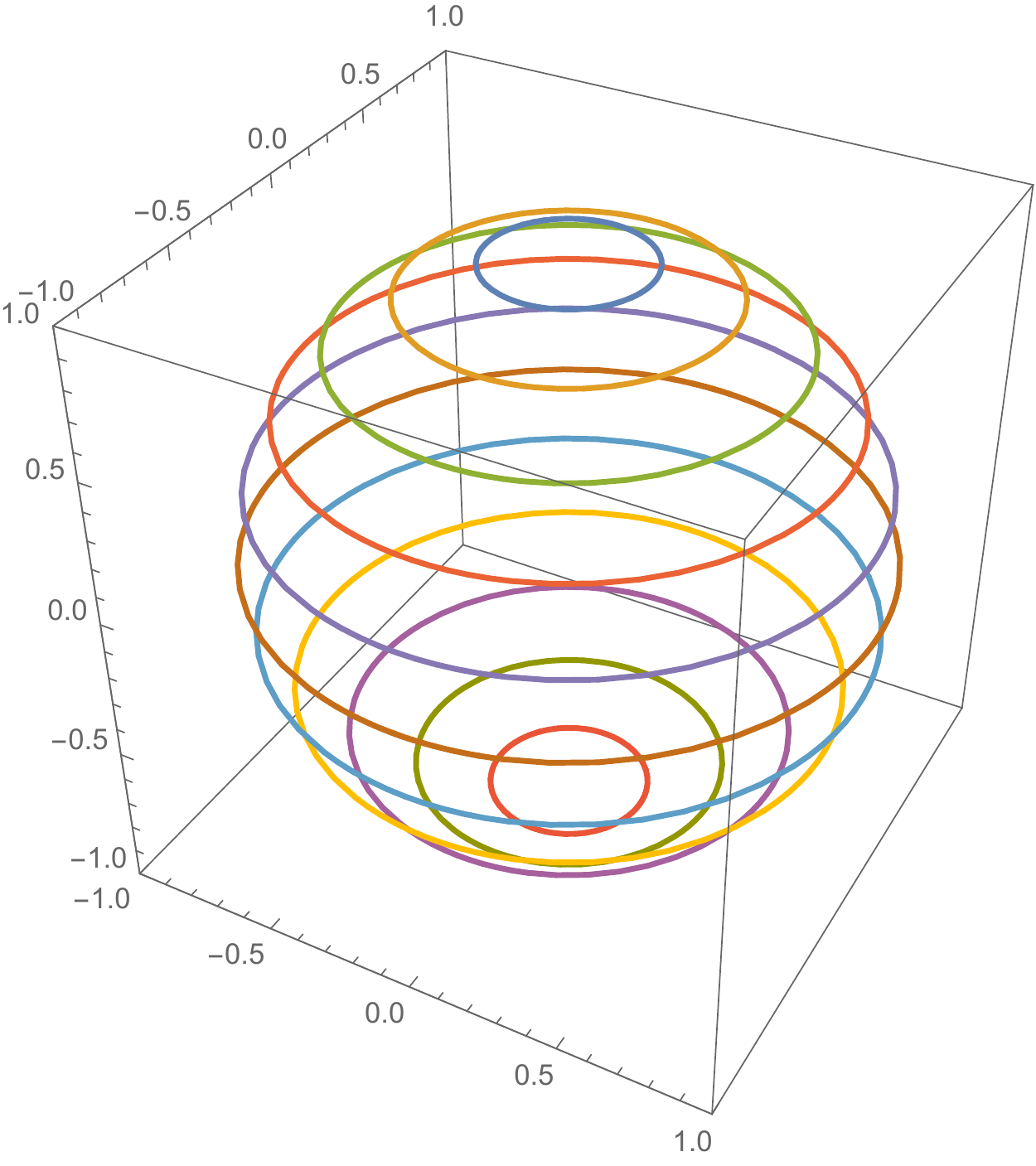}
\includegraphics[scale=0.3]{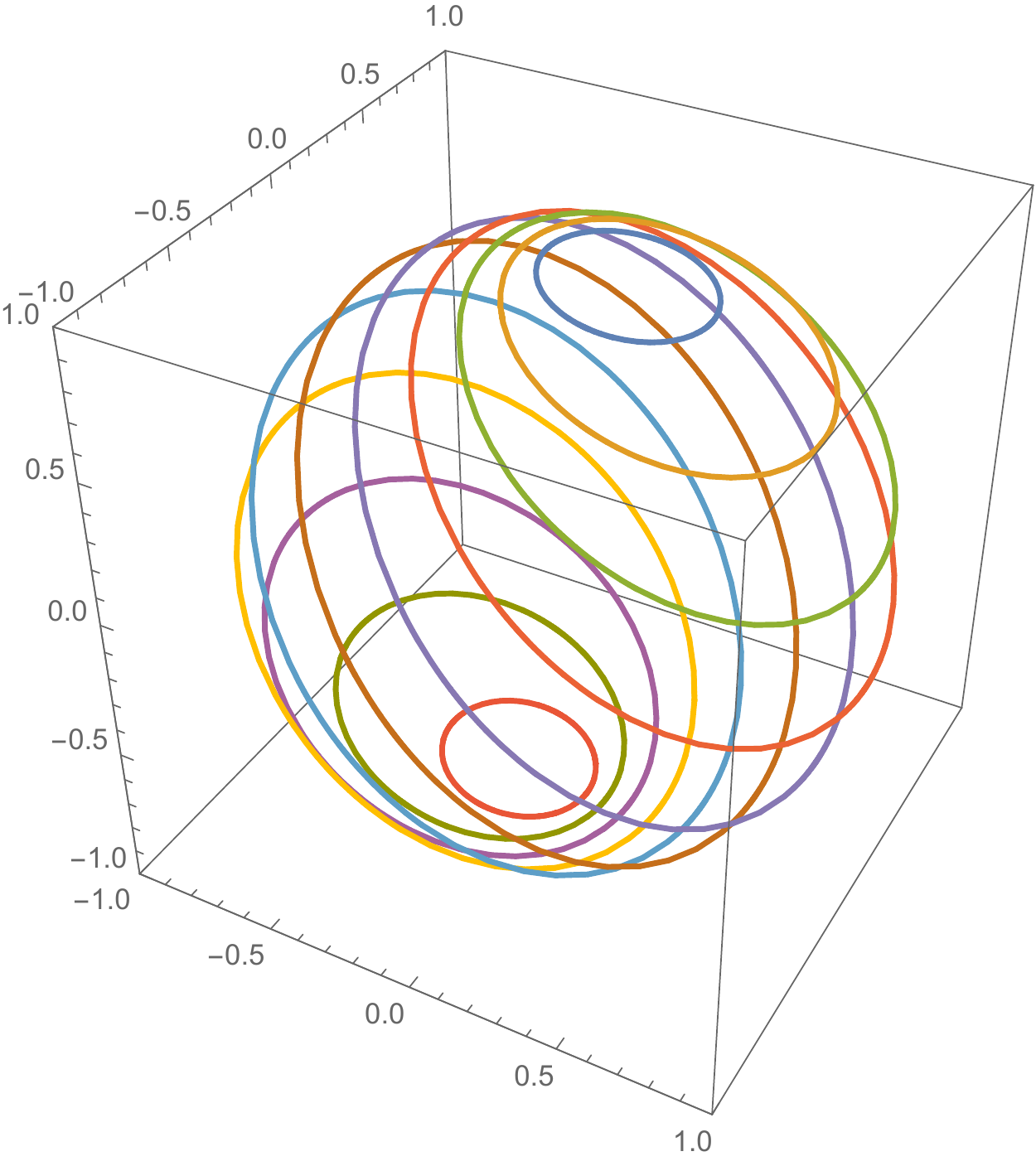}
\includegraphics[scale=0.3]{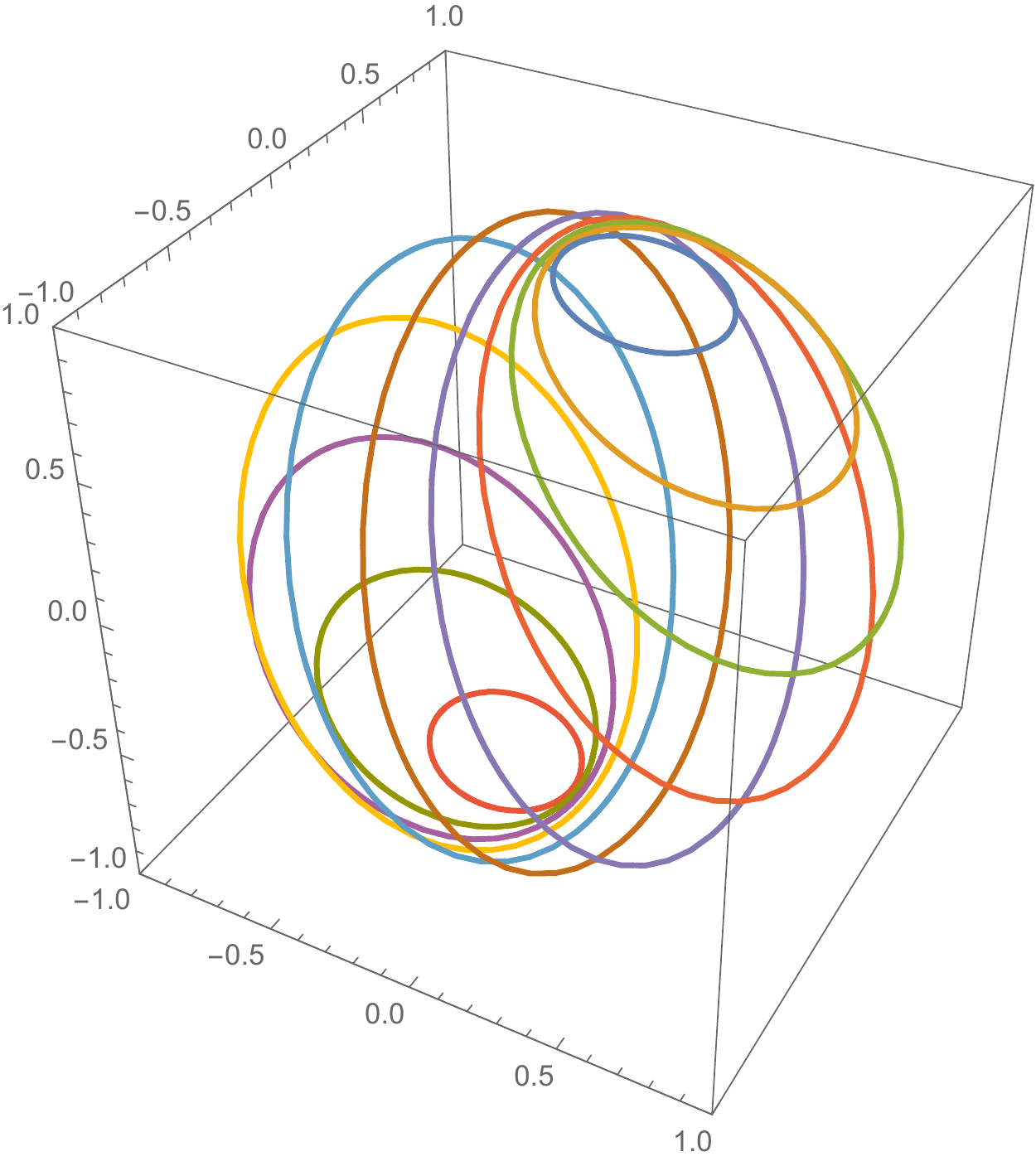}
    \caption{$C_{R}(p)$ shown for various $R$ with $b=0,0.75$ and $0.95$, respectively. }
  \end{figure}
 Since $x(p)=b\sin R$, $y(p)=0$, we conclude from Theorem \ref{thrm:5.7} that $C_{R}(p)$ also satisfies $r_{s}=b\sin\theta$. It is clear that the points $(\sin (f(R)\pm R),0,\cos(f(R)\pm R))$ lie on $C_{R}(p)$. Define
  \begin{equation*}
    g_{1}(R)=f(R)+R\spa g_{2}(R)=f(R)-R,
  \end{equation*}
  Since $f'(R)=b\cos R/\cos f(R)$, we obtain
  \begin{equation*}
    f'(R)^{2}=\frac{b^{2}\cos^{2} R}{1-b^{2}+b^{2}\cos^{2}R}<1,
  \end{equation*}
  so $g_{1}(R)$ is increasing, and similarly, $g_{2}(R)$ is decreasing. Since $g_{1}(\pi)=\pi$ and $g_{2}(\pi)=-\pi$, we conclude that $g_{1}$ is surjective onto $(0,\pi)$ and $g_{2}$ is surjective onto $(-\pi,0)$.

  Now let $\mu\in \gamma$ be the point which maximizes the radial
  function $r|_{\gamma}$. Then $r_{s}(\mu)=0$, so $\sin\theta(\mu)=0$,
  so $\mu=(\sin \alpha,0,\cos\alpha)$ for some $\alpha$. Since $\mu\ne
  (0,0,\pm 1)$, we can take $\alpha\in (-\pi,0)\cup(0,\pi)$. Thus
  there is some $R$ such that $\alpha=g_{1}(R)$ or $\alpha=g_{2}(R)$,
  and thus $\mu\in C_{R}(p)$. Since $C_{R}(p)$ also satisfies
  $r_{s}=b\sin\theta$, we may invoke Theorem \ref{thrm:5.4} to conclude that $\gamma=C_{R}(p)$.
\end{proof}

A similar argument to the $\S{2}$ case works for the $\H{2}$ case (we
found it useful to work in the hyperboloid model $\H{2}\subset
\R{3}$), and we obtain
\begin{thrm}\label{thrm:5.9}
  Let $N=\R{2},k^{-1}\S{2}$ or $k^{-1}\H{2}$. If $\gamma_{0}\subset N$ is a
  $C^{1}$ radial graph satisfying $r_{s}=a\cos\theta+b\sin\theta$, then $\gamma_{0}$ is a circle (which may be translated).
\end{thrm}
\begin{thrm}\label{thrm:5.10}
  Let $N$ be a warped product space, let $\gamma_{0}\subset N$ be a
  piecewise $C^{1}$ Lipshitz radial graph satisfying
  $L[\gamma_{0}]=F(A[\gamma_{0}])$, and suppose
  $0\le \varphi'\varphi'-\varphi\varphi''\le 1$. Then, considering $\gamma_{0}$
  as lying in a subset of one of the spaceforms, $\gamma_{0}$ is a
  circle.
\end{thrm}
\begin{proof}
  We have already proved the case where
  $\varphi'\varphi'-\varphi\varphi''\not\equiv 1$ in Theorem \ref{thrm:5.2}, so we assume that
  $\varphi'\varphi'-\varphi\varphi''\equiv 1$. Using the same ``cut and reflect'' technique used in Theorem \ref{thrm:4.3} and
  Theorem \ref{thrm:5.2}, we may cut $\gamma_{0}$ (along an axis
  $\alpha\in \S{1}$) into two bilaterially symmetric halves
  $\gamma_{1}$ and $\gamma_{2}$. Following Theorem \ref{thrm:5.2}, the solution
  $\gamma_{1}(t)$ (guaranteed by Theorem \ref{thrm:5.1}) has constant length, and
  thus (by Corollary \ref{cor:2.4}) $r_{s}(t)=a(t)\cos\theta+b(t)\sin\theta$
  holds along the flow, and so (by Theorem \ref{thrm:5.9}) $\gamma_{1}(t)$ is a
  circle for all positive $t$. Let $r_{1}$ be the (constant) radius of
  $\gamma_{1}(t)$, let $C_{0}$ be the (translated) circle of radius
  $r_{1}$ fixed under the reflection through $\alpha$ which contains
  the points $(\rho_{0}(\pm \alpha),\pm\alpha)$. The centre of $C_{0}$
  is uniquely determined by its axis of symmetry $\alpha\in \S{1}$,
  its radius $r_{1}$, and the points
  $(\rho_{0}(\pm \alpha),\pm\alpha)$. The centre of $\gamma_{1}(t)$ is
  also determined by its axis of symmetry, its radius, and the point
  $(\rho_{1}(\pm\alpha,t),\alpha)$, and by convergence
  $\rho_{1}(\pm\alpha,t)\to \rho_{0}(\pm\alpha)$, we deduce that
  $\gamma_{1}(t)$ converges to $C_{0}$ uniformly as $t\searrow 0$. We
  deduce that $\gamma_{1}=C_{0}$, and similarly, $\gamma_{2}=C_{0}$,
  and thus by the construction of $\gamma_{1}$ and $\gamma_{2}$, we
  conclude $\gamma_{0}=C_{0}$, as desired.
\end{proof}
Combining Theorems \ref{thrm:4.3}, \ref{thrm:5.2} and \ref{thrm:5.10}, we conclude the isoperimetric
inequality Theorem 1.2 stated in the introduction.
\section{Conclusion}
It seems to be a general phenomenon that many sophisticated tools used
in higher dimensions $n>1$ (such as the Minkowski identity
\eqref{eq:1.3}) cannot be used when $n=1$. The symmetry argument used
in Theorem 4.1 replaced the use of the Minkowski identity, but
necessitated a more complicated argument to deal with the
non-symmetric case. 

As mentioned in section 2.1, the restriction
$\varphi'\varphi'-\varphi\varphi''\le 1$ is sharp - without it the
isoperimetric inequality is guaranteed to fail. However, the
restriction $\varphi'\varphi'-\varphi\varphi''>0$ is not sharp (we
used this inequality to prove the convergence to a circle). For
example, the cylinder $\S{1}\cross [0,1]$ satisfies
$\varphi'\varphi'-\varphi\varphi''\equiv 0$, and it is easy to explicitly
show that it does satisfy the isoperimetric inequality.

\proof[Acknowledgements]
Many thanks to my supervisor, Pengfei Guan, for providing guidance,
insight, and many office hours to answer my questions. I would also
like to thank the McGill Mathematics department for their support. The
author was supported by NSERC undergraduate student research award.

\bibliographystyle{alpha}
\bibliography{citations}

\end{document}